\documentclass[11pt,psamsfonts]{amsart}
\usepackage{amsmath}
\usepackage{amsthm}
\usepackage{amssymb}
\usepackage{amscd}
\usepackage{amsfonts}
\usepackage{amsbsy}
\usepackage{graphicx}
\usepackage[dvips]{psfrag}
\usepackage{array}
\usepackage{color}
\usepackage{epsfig}
\usepackage{url}
\usepackage{ulem}

\newcolumntype{L}{>{\displaystyle}l}
\newcolumntype{C}{>{\displaystyle}c}
\newcolumntype{R}{>{\displaystyle}r}

\def \fim {{\hfill $\Box$}}
\def \np {\noindent}

\def \g {\gamma}
\def \a {\alpha}
\def \b {\beta}
\def \R {\mathbb{R}}
\def \t {\times}

\def \> {\geq}
\def \< {\leq}

\def \th {\theta}

\def \lb {\lambda}

\def \rt {\rightarrow}
\def \Rl {\mathcal{R}}
\def \ph {\varphi}
\def \Th {\Theta}

\newtheorem {theorem} {Theorem}
\newtheorem {proposition} [theorem] {Proposition}

\newtheorem {lemma} [theorem] {Lemma}
\newtheorem {definition} {Definition}

\newtheorem {remark} {Remark}

\begin{document}

\title[Functions on a swallowtail]
{Functions on a swallowtail}

\author[A.P. Francisco]
{Alex Paulo Francisco}

%
%
%




\maketitle

\section*{\bf Abstract}

We classify submersions from $(\R^3,0)$ to $(\R,0)$ up to
diffeomorphisms which preserve the swallowtail and use this
classification to study its flat geometry. The flat geometry is
derived from the contact of the swallowtail with planes, which is
measured by the singularities of the height function.

\smallskip

%


\section{\bf Introduction}


A \textit{swallowtail} is the image of a germ
$g:(\R^2,0)\rightarrow(\R^3,0)$ that is $\mathcal{A}$-equivalent to
$f(x,y) = (x,-4y^3-2xy,3y^4+xy^2)$, that is, there exist germs of
diffeomorphisms $\phi$ and $\psi$ such that $g = \psi\circ
f\circ\phi^{-1}$. We refer to the swallowtail parametrised by $f$ as
the \textit{standard swallowtail} (see Figure \ref{fig:swallow}) and
to the swallowtail parametrised by $g$ as the \textit{geometric
swallowtail}. In \cite{forma normal} is given a normal form of a
geometric swallowtail obtained using changes of coordinates in the
source and isometries in the target.

\begin{figure}[h!]
\centering
\includegraphics[scale = 1.2]{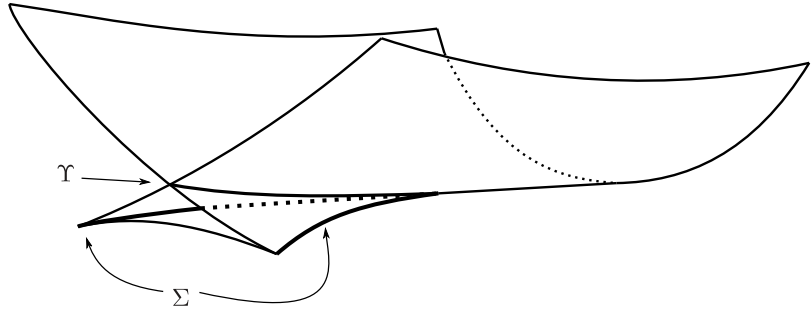}
\caption{\small{The swallowtail, its singular curve $\Sigma$ and its
double point curve $\Upsilon$.}} \label{fig:swallow}
\end{figure}

Swallowtail surfaces arise in a natural way. For instance, the focal
sets, duals and discriminants of curves and surfaces in the
Euclidean space $\R^3$ can have swallowtail singularities (see for
example \cite{Arnold-livro}, \cite{Bruce-Giblin-ex},
\cite{Bruce-Wilkinson}, \cite{S-U-Y}). Hence it is important to
study their differential geometry.

In this paper, we classify germs of submersions
$f:(\R^3,0)\rt(\R,0)$ up to diffeomorphisms in the source which
preserve the swallowtail. We also study the flat geometry of a
swallowtail which is derived from its contact with planes (flat
objects). This contact is measure by the singularities of the height
function on the swallowtail.

This work is part of an ongoing study of the geometry of singular
surfaces from Singularity Theory viewpoint (see for example
\cite{Martin-Yutaro}, \cite{BW}, \cite{Fabio-Farid},
\cite{Fukui-Masaru}, \cite{H-H-N-U-Y}, \cite{H-H-N-S-U-Y},
\cite{H-N-U-Y}, \cite{Juan-Farid}, \cite{Oliver}, \cite{Raul-Farid},
\cite{West} for the cross-cap, \cite{Bruce-Wilkinson},
\cite{K-R-S-U-Y}, \cite{Lu-Juan}, \cite{Lu-kentaro},
\cite{Lu-kentaro2}, \cite{N-U-Y}, \cite{Cuspidal}, \cite{S-U-Y},
\cite{Teramoto}, \cite{Wilkinson} for the cuspidal edge, \cite{forma
normal} for the swallowtail and \cite{cuspidal crosscap} for the
cuspidal cross-cap).

We follow the approach in \cite{BW} : we fix the standard
swallowtail $X = f(\R^2,0)$ and consider its contact with fibres of
submersions. (See \S 3.3 for details)

The paper is organized as follows. In \S 2 we give some concepts and
results on classification of germs of functions on an analytic
variety. In \S 3 we give some properties of the standard swallowtail
and classify submersions from $(\R^3,0)$ to $(\R,0)$ up to changes
of coordinates in the source that preserves the standard
swallowtail, in \S 4 we obtain the discriminants of versal
unfoldings of each normal form obtained in the classification and
analyze the contact between the zero fiber and the standard
swallowtail in each case. We use in \S 5 the classification in \S 3
to study the flat geometry of a geometric swallowtail. We based on
\cite{BW} and follow its approach.

The background material on singularity theory that we refer to the
reader to \cite{Martinet}, \cite{Wall-livro} and on its application
in the differential geometry to \cite{BG-livro}, \cite{Livro-Farid}.

This paper is part of the PhD Thesis work of the author under
supervision of Farid Tari. For more details see \cite{Tese}.

\section{\bf Functions on analytic varieties}

In this section we review some concepts and results from
\cite{Bruce-trans}, \cite{BR}, \cite{BW}, \cite{Damon} and
\cite{Cuspidal} which are useful tools for classifying functions on
analytic varieties.

Let $\mathcal{E}_n$ be the local ring of germs of smooth functions
$(\R^n,0)\rightarrow\R$ and $\mathcal{M}_n$ its unique maximal
ideal.

Let $(X,0)\subset(\R^n,0)$ be a germ of a reduced analytic
subvariety of $\R^n$ at $0$. We say that a germ of diffeomorphism
$\ph:(\R^n,0)\rt(\R^n,0)$ \textit{preserves $X$} if $\ph(X)$ and $X$
are equal as germs at $0$, that is, $(\ph(X),0) = (X,0)$. The group
of such diffeomorphisms is a subgroup of $\Rl$ and is denoted by
$\Rl(X)$.

Given two germs $f,g\in\mathcal{E}_n$, we say that they are
\textit{$\Rl(X)$-equivalents} if there exists a germ of
diffeomorphism $\ph\in\Rl(X)$ such that $g\circ\ph^{-1} = f$.

We denote by $\Th(X)$ the $\mathcal{E}_n$-module of germs of vector
fields tangent to $X$ at $0$. We define $\Th(X)\cdot f = \{\xi \cdot
f \in\mathcal{E}_n \, |\, \xi\in\Th(X) \, , \xi(0) = 0\}$, which is
an $\mathcal{E}_n$-module.

Let $\Th_1(X) = \{\xi\in\Th(X) \, |\, j^1\xi = 0\}$ which is an
$\mathcal{E}_n$-module. If we integrate the vector fields in
$\Th_1(X)$ we obtain a group denoted by $\Rl_1(X)$, which is the set
of germs of diffeomorphisms in $\Rl(X)$ with $1$-jets is the
identity. We also can define the subgroup $\Rl_k(X)$ of germ of
diffeomorphism at $\Rl(X)$ with $k$-jets is the identity. It is a
normal subgroup of $\Rl(X)$ and, consequentially, we can define the
group $\displaystyle\Rl^{(k)}(X) = \frac{\Rl(X)}{\Rl_k(X)}$. The
elements of $\Rl^{(k)}(X)$ are $k$-jets of elements of $\Rl(X)$. The
action of $\Rl(X)$ on $\mathcal{M}_n$ induces an smooth action of
the group $\Rl^{(k)}(X)$ on the $k$-jet space of function germs
$J^{k}(n,1)$.

For $f\in\mathcal{E}_n$ the tangent spaces to the $\Rl(X)$ and
$\Rl_1(X)$-orbits of $f$ are, respectively
$$L\Rl(X)\cdot f = \Th(X)\cdot f \ \ \ \ {\rm and} \ \ \ \ \ \ \ L\Rl_1(X)\cdot f = \Th_1(X)\cdot f.$$

The tools for classifying germs of functions $(\R^n,0)\rt\R$, up to
the $\Rl(X)$-equivalence, are generalizations of the classical
results about the action of $\Rl$ over $\mathcal{E}_n$. The group
$\Rl(X)$ is a Damon's geometric subgroup (\cite{Damon}), so the
theorems of on versal deformations and finite determinacy apply to
this setting.

\begin{definition}
A germ $f:(\R^n,0)\rt(\R,0)$ is \textit{$k-\Rl(X)$-determined} if
every germ of a function with the same $k$-jet as $f$ is
$\Rl(X)$-equivalent to $f$. We say that $f$ is
\textit{$\Rl(X)$-finitely determined} if $f$ is
$k-\Rl(X)$-determined for same $k\in\mathbb{N}^*$.
\end{definition}

\begin{theorem}{\rm (\cite{Damon})}
Consider a germ $f:(\R^n,0)\rt(\R,0)$. If there exists
$k\in\mathbb{N}^*$, such that $$\mathcal{M}_n^k\subset L\Rl(X)\cdot
f,$$ then $f$ is $(k+1)-\Rl(X)$-determined.
\end{theorem}

We define the  \textit{extended pseudo-group} of diffeomorphisms
preserving $X$, denoted by $\Rl_e(X)$, as being the pseudo-group
obtained by integrating the vector fields $\xi\in\Th(X)$, but
excluding the condition $\xi(0) = 0$. Hence, for $f\in\mathcal{E}_n$
the extended tangent space to the $\Rl_e(X)$-orbit of $f$ is
$L\Rl_e(X)\cdot f = \{\xi\cdot f\in \mathcal{E}_n \, | \,
\xi\in\Theta(X)\}$.

Note that when $X$ is a swallowtail, every vector field vanish at
the origin, that is, in this case $\Rl_e(X) = \Rl(X)$.

The $\Rl(X)$-classification of germs finitely determined is carried
out inductively on the jet level. The method used here is that of
the complete transversal \cite{Bruce-trans} adapted for the
$\Rl(X)$-action in \cite{BW}.

\begin{theorem} \textbf{Complete Transversal}\label{teo:trans.comp}
Let $f:(\R^n,0)\rt(\R,0)$ be a smooth germ and $\{h_1,...,h_r\}$ a
collection of homogeneous polynomials of degree $k+1$ such that
$$\mathcal{M}_n^{k+1}\subset L\Rl_1(X)\cdot f + \R\cdot\{h_1,...,h_r\} +
\mathcal{M}_n^{k+2}.$$ Then any germ $g:(\R^n,0)\rt(\R,0)$ with
$j^kg(0) = j^kf(0)$ is $\Rl_1(X)$-equivalent to a germ of the form
$$f(x) + \sum_{i=1}^r \lb_ih_i(x) + \ph(x),$$ where
$\ph(x)\in\mathcal{M}_n^{k+2}$ and $\lb_i\in\R$. The real vector
space $T = \R\cdot\{h_1,...,h_r\}$ is called by a \textit{complete
$(k+1)$-transversal of $f$}.
\end{theorem}

\begin{proposition}\label{corol:det}

\np {\rm{(i)}} A germ $f\in \mathcal{M}_n$ is
$k-\Rl_1(X)$-determined if and only if $$\mathcal{M}_n^{k+1}\subset
L\Rl_1(X)\cdot f + \mathcal{M}_n^{k+2}.$$

\np {\rm{(ii)}} In particular, if every vector field in $\Th(X)$
vanishes at the origin and $$\mathcal{M}_n^{k+1}\subset L\Rl(X)\cdot
f + \mathcal{M}_n^{k+2},$$ then $f$ is $(k+1)-\Rl(X)$-determined.
\end{proposition}
\begin{proof}
This is a consequence of Theorem \ref{teo:trans.comp} and Theorem
2.5 in \cite{BDW} applied to
our setting.\fim\\
\end{proof}

An \textit{$s$-parameter deformation of $f\in \mathcal{E}_n$} is a
family of germs of functions\linebreak
$F:(\R^n\t\R^s,(0,0))\rightarrow(\R,0)$ such that $F_0(x) = F(x,0) =
f(x)$. An $s$-parameter deformation $F$ is said to be
\textit{$P$-$\mathcal{R}^+(X)$-induced} from an $r$-parameter
deformation $G$ if there exist a germ
$\phi:(\R^n\t\R^s,(0,0))\rightarrow(\R^n\t\R^r,(0,0))$ of the form
$\phi(x,u) = (\varphi(x,u),\psi(u))$ and a germ of a function
$c:(\R^s,0)\rightarrow\R$ such that $F(x,u) = G(\phi(x,u)) + c(u)$.
When $\phi$ is a germ of a diffeomorphism we say that $F$ and $G$
are \textit{$P$-$\Rl^+(X)$-equivalent} (see for example
\cite{BG-livro} for the notion of ($p$)-unfoldings).

We say that a deformation $F$ of $f$ is an \textit{$\Rl^+(X)$-versal
deformation of $f$} if any other deformation of $f$ is
$P$-$\Rl^+(X)$-induced from $F$.

\begin{proposition}{\rm (\cite{BW})}
An $s$-parameter deformation $F$ of a germ of a function $f$ on $X$
is an $\Rl^+(X)$-versal deformation if and only if
$$L\Rl_e(X)\cdot f  + \R\cdot\{1,\dot{F_1}, ..., \dot{F_s}\} = \mathcal{E}_n,$$
where $\dot{F_i} = \frac{\partial F}{\partial u_i}(x,0)$, for $i =
1,...,s$.
\end{proposition}

We define the $\Rl_e^+(X)$-codimension of $f$ as $\displaystyle
cod(f,\Rl_e^+(X)) = dim_{\R}(\frac{\mathcal{M}_n}{L\Rl_e(X)\cdot
f})$. It is the least number of parameters needed to have an
$\Rl^+(X)$-versal deformation of $f$.

Another important tool in the classification is Mather's Lemma.

\begin{lemma}{\rm (\cite{Mather})} \textbf{Mather's Lemma}\label{lem:mather}
Let $\a:G\t M\rt M$ be a smooth action of a Lie group $G$ over a
smooth manifold $M$, and let $V$ be a connected submanifold of $M$.
Then the necessary and sufficient conditions for $V$ been in a
single orbit are the following:
\begin{itemize}
\item[(i)] $T_vV\subset T_v(G.v)$, for every $v\in V$.

\item[(ii)] $dim(T_v(G.v))$ is independent of $v\in V$.
\end{itemize}
\end{lemma}

\section{\bf Classification of functions on a swallowtail}

In this section, we shall use the results in \S 2 to classify smooth
germs of functions from $(\R^3,0) \rightarrow (\R,0)$ up to changes
of coordinates in the source which preserve the standard
swallowtail. Note that when $X$ is a swallowtail, every vector field
vanish at the origin, that is, in this case $\Rl_e(X) = \Rl(X)$.

We consider here $X$ to be the standard swallowtail parametrised by
$f(x,y) = (x,-4y^3-2xy,3y^4+xy^2)$ or with equation
$$16u^4w-4u^3v^2-128u^2w^2+144uv^2w-27v^4+256w^3 = 0.$$
We called this germs \textit{functions on a swallowtail}. Note that
the function $f$ is a parametrisation of the discriminant set of the
$\Rl$-versal deformation $F(t,u_0,u_1,u_2) = t^4 + u_2t^2 + u_1t +
u_0$ of the $A_3$-singularity $t^4$.

\begin{proposition}\label{prop:base} {\rm(\cite{Bruce-disc})}
The $\mathcal{E}_3$-module of germs at the origin of vector fields
in $\R^3$ tangents to the standard swallowtail is generated by the
vector fields $\th_1, \th_2$ and $\th_3$ with
\begin{itemize}
\item[] $\displaystyle\theta_1 = 2u\frac{\partial}{\partial u} + 3v\frac{\partial}{\partial
v} + 4w\frac{\partial}{\partial w}$,

\item[] $\displaystyle\theta_2 = 6v\frac{\partial}{\partial u} + (8w-2u^2)\frac{\partial}{\partial
v} -uv\frac{\partial}{\partial w}$,

\item[] $\displaystyle\theta_3 = (16w-4u^2)\frac{\partial}{\partial u} -8uv\frac{\partial}{\partial
v} -3v^2\frac{\partial}{\partial w}$.
\end{itemize}
\end{proposition}

Integrating the linear parts of $\th_1, \th_2, \th_3$ in Proposition
\ref{prop:base}, gives the followings $1$-jets of changes of
coordinate in $\Rl(X)$
\begin{itemize}
\item[] $h_1(u,v,w) = (e^{2\lb}u,e^{3\lb}v,e^{4\lb}w)$,

\item[] $h_2(u,v,w) = (u+3\b v,v + 4\g w,w)$,

\item[] $h_3(u,v,w) = (u+\a w,v,w)$,
\end{itemize}
with $\a,\b,\g,\lb\in \R$.

Consider the $1$-jet $j^1f = au + bv + cw$ of a submersion $f$,
with $a, b$ or $c$ non-zero.\\

%
%


\begin{proposition}\label{prop:class}
The $\Rl^{(1)}(X)$-orbits of submersions $f:(\R^3,0)\rightarrow
(\R,0)$ are $ \pm u, \, v, \, \pm w$.
\end{proposition}
\begin{proof}
The proof immediately follows considering the 1-jets of
diffeomorphisms in $\Rl(X)$. \fim
\end{proof}

Now we investigate each case in Proposition \ref{prop:class}.

\begin{lemma}\label{lem:01}
The germ $g(u,v,w) = \pm u$ is $1-\Rl(X)$-determined and has
$\Rl_e^+(X)$-codimension $0$.
\end{lemma}
\begin{proof}
We have
$$L\Rl(X)\cdot g = \mathcal{E}_3\cdot\{u,v,4w-u^2\}  = \mathcal{M}_3$$
and the result follows. \fim
\end{proof}

\begin{lemma}\label{lem:02}
Any $\Rl(X)$-finitely determined germ in $\mathcal{E}_3$ with
$1$-jet $\Rl^{(1)}(X)$-equivalent to $v$ is $\Rl(X)$-equivalent to
$v + au^{k+1}$ for some $k\> 1$ and $a\neq 0$. The germ $v +
au^{k+1}$, $a\neq 0$, is $(k+1)-\Rl(X)$-determined and has
$\Rl_e^+(X)$-codimension $k$.
\end{lemma}
\begin{proof}
Observe that the germ $v$ is not $\Rl(X)$-finitely determined. We
prossed by induction on the $k$-jets ($k\> 1$) of germs $g$ with
$1$-jet $v$.

Firstly, we find a complete $(k+1)$-transversal of $g(u,v,w) = v$.

Note that $$L\Rl_1(X)\cdot g = \mathcal{M}_3\cdot\{v,4w-u^2\} =
\mathcal{E}_3\cdot\{uv, v^2, vw, 4uw - u^3, 4w^2 - u^2w\}.$$

Hence, $$\mathcal{M}_3^{(k+1)}\subset L\Rl_1(X)\cdot g + \R\cdot\{
u^{k+1}\} + \mathcal{M}_3^{(k+2)},$$ so $T = \R\cdot\{u^{k+1}\}$ is
a complete $(k+1)$-transversal of $g$. Then, by Theorem
\ref{teo:trans.comp}, any $(k+1)$-jet with $k$-jet equal to $v$ is
$\Rl_1(X)$-equivalent to $v + au^{k+1}, \ a\in \R$.

For $a\neq 0$, using Proposition \ref{corol:det}, we can conclude
that the germ $\overline{g}(u,v,w) = v + au^{k+1}$ is
$(k+2)-\Rl(X)$-determined. However, we can use Theorem
\ref{teo:trans.comp} and Lemma \ref{lem:mather} to conclude that
$\overline{g}$ is, in fact, $(k+1)-\Rl(X)$-determined.

We have $\displaystyle\frac{\mathcal{M}_3}{L\Rl(X)\cdot
\overline{g}} = \R\cdot\{u,u^2,...,u^k, u^{k+1}\}$ which implies
that the $\Rl_e^+(X)$-codimension of $\overline{g}$ is $k+1$ and the
codimension of the stratum of this singularity is $k$.\fim
\end{proof}

\begin{lemma}\label{lem:03}
Any $\Rl(X)$-finitely determined germ in $\mathcal{E}_3$ with
$1$-jet $\Rl(X)$-equivalent to $\pm w$ and $\Rl_e^+(X)$-codimension
$\< 2$ is $\Rl(X)$-equivalent to $\pm w + au^2 + bu^3$, with $a\neq
0, \pm \frac{1}{12}, \pm\frac{1}{4}$ and $b\neq 0$. Furthermore, the
germ $\pm w + au^2 + bu^3$, with $a$ and $b$ in the previous
conditions, is $3-\Rl(X)$-determined and has
$\Rl_e^+(X)$-codimension $2$ (on the stratum).
\end{lemma}
\begin{proof}
For $g(u,v,w) = \pm w$ we have
$$L\Rl(X)\cdot g = \mathcal{E}_3\cdot\{w, uv, v^2\},$$
so $g$ is not $\Rl(X)$-finitely determined. We proceed by induction
on the $k$-jets of germs with $1$-jet $\pm w$.

Note that $$\mathcal{M}_3^2\subset L\Rl_1(X)\cdot g + \R\cdot\{u^2,
uv, v^2\} + \mathcal{M}_3^3,$$ so $T = \R\cdot\{u^2, uv, v^2\}$ is a
complete $2$-transversal of $g$. Hence, any $2$-jet with $1$-jet
equal to $\pm w$ is $\Rl_1(X)$-equivalent to $g(u,v,w) = \pm w +
au^2 + buv + cv^2$, with $a,b,c\in\R$.

When $a\neq 0$, using the linear change of coordinates $h_2$ with
$\g = 0$ and $\b = \frac{-b}{6a}$, we obtain $g(h_2(u,v,w)) = \pm w
+ au^2 + c'v^2$.

We can show, using Mather's Lemma, that $\pm w + au^2 + c'v^2$ is
$\Rl(X)$-equivalent to $\pm w + au^2$.

Consider $f(u,v,w) = \pm w + au^2$, with $a\neq 0$. Then
$$L\Rl(X)\cdot f = \mathcal{E}_3\cdot\{ au^2 \pm w, (12a \mp 1)uv,
32auw - 8au^3 \mp 3v^2\}.$$

A complete $3$-transversal is given by
$$T = \left\{\begin{array}{lll}
        \R\cdot\{u^3\} &  {\rm if} & a\neq \pm \frac{1}{12} \\
        \R\cdot\{u^3,u^2v\} &  {\rm if} & a = \pm \frac{1}{12}
      \end{array}.\right.$$

Therefore, when $a\neq 0,\pm\frac{1}{12}$, any $3$-jet with $2$-jet
equal to $\pm w + au^2$ is $\Rl_1(X)$-equivalent to $\pm w + au^2 +
bu^3$, $b\in\R$.

For $\overline{f}(u,v,w) = \pm w + au^2 + bu^3$, with $a\neq 0,
\pm\frac{1}{12}$, we have
$$\mathcal{M}_3^4\subset L\Rl_1(X)\cdot\overline{f} + \mathcal{M}_3^5,$$
if and only if $a\neq \pm\frac{1}{4}$, that is, by Proposition
\ref{corol:det}, $\overline{f}$ is $3-\Rl(X)$-determined if and only
if $a\neq \pm\frac{1}{4}$.

Furthermore,
$$\frac{\mathcal{M}_3}{L\Rl(X)\cdot\overline{f}} = \left\{\begin{array}{lll}
       \R\cdot\{u,v,u^2,u^3\} &  {\rm if} & b\neq 0 \\
        \R\cdot\{u,v,u^2,u^3,v^2\} &  {\rm if} & b = 0
      \end{array}\right.$$
which implies that the $\Rl_e^+(X)$-codimension of the stratum of
the singularity of $\overline{f}$ is $2$ if $b\neq 0$ and $4$ if $b
= 0$.

When $a = 0$, any $\Rl(X)$-finitely determined germ in
$\mathcal{E}_3$ with 2-jet $\Rl(X)$-equivalent to $\pm w + buv +
cv^2$ has $\Rl_e^+(X)$-codimension $> 2$. \fim
\end{proof}

\begin{theorem}\label{Teo:class}
Let $X$ be the swallowtail parameterised by $f(x,y) =\linebreak
(x,-4y^3-2xy,3y^4+xy^2)$. Denote by $(u,v,w)$ the coordinates in the
target. Then any germ $g:(\R^3,0)\rt(\R,0)$ of an $\Rl(X)$-finitely
determined submersion with $\Rl_e^+(X)$-codimension $\< 2$ of the
stratum in the presence of moduli is $\Rl(X)$-equivalent to one of
the germs in Table \ref{tab:clas}.

\begin{table}[h]
\caption{$\Rl_e^+(X)$-codimension $\< 2$ germs of
submersions.}\label{tab:clas}
\begin{tabular}{|l|l|l|}
  \hline
  \textbf{Normal form} & \textbf{$cod(f,\Rl_e^+(X))$} & \textbf{$\Rl^+(X)$-versal deformation} \\
  \hline
  $\pm u$ & {\rm 0} & $\pm u$ \\
  $v + au^2, a\neq 0$ & {\rm 1} & $v + au^2 + a_1u$ \\
  $v + au^3, a\neq 0$ & {\rm 2} & $v + au^3 + a_1u + a_2u^2$ \\
  $\pm w + au^2 + bu^3, a\neq 0,\pm \frac{1}{12}, \pm \frac{1}{4}; b\neq 0$ & {\rm 2} & $\pm w + au^2 + bu^3 + a_1u + a_2v$ \\
  \hline
\end{tabular}
\end{table}

\end{theorem}
\begin{proof}
The proof follows from Proposition \ref{prop:class} and Lemmas
\ref{lem:01}, \ref{lem:02}, \ref{lem:03}.\fim
\end{proof}

\begin{remark}
The $\mathcal{K}(X)$-classification of germs of submersions
$(\R^3,0)\rt (\R,0)$ of $\mathcal{K}_e(X)$-codimension $\< 2$ can be
obtained from Theorem \ref{Teo:class} by setting $a = \pm 1$.
Furthermore, we observe that if we are interested in the fibers of
these submersions, then both classifications can be used, since the
fibers will be diffeomorphics.
\end{remark}

\section{\bf The geometry of functions on a swallowtail}

The standard swallowtail has equation
$16u^4w-4u^3v^2-128u^2w^2+144uv^2w-27v^4+256w^3 = 0$. By Shafarevich
\cite{Cone tangente}, if $X$ is an irreducible affine variety in
$\R^n$ defined by the ideal $I$ then the equations of the tangent
cone of $X$ are the lowest degree terms of the polynomials in $I$.
Therefore, the tangent cone to the standard swallowtail is the
repeated plane $w^3 = 0$. The tangential line of the standard
swallowtail at the origin is the line with direction $(1,0,0)$
passing through the origin. The germ $f(x,y) =
(x,-4y^3-2xy,3y^4+xy^2)$ is singular along a curve $\Sigma$
parametrised by $\a(t) = f(-6t^2,t) = (-6t^2,8t^3,-3t^4)$.
Furthermore $f$ has a double point curve $\Upsilon$ parametrised by
$\b(t) = f(-2t^2,t) = (-2t^2,0,t^4)$ which ends at the swallowtail
point. See Figure \ref{fig:swallow}.

We study here the discriminants of the singularities given in
Theorem \ref{Teo:class}. Let $g:(\R^3,0)\rt(\R,0)$ be a germ on $X =
f(\R^2,0)$ and $F:(\R^3\t\R^2,(0,0))\rt(\R,0)$ be a deformation of
$g$. We consider the families $G(x,y,a_1,a_2) = F(f(x,y),a_1,a_2)$,
$H_1(t,a_1,a_2) = F(\a(t),a_1,a_2)$ and $H_2(t,a_1,a_2) =
F(\b(t),a_1,a_2)$.

The discriminant of the family $G$ is the set
$$\mathcal{D}_1(F) = \{(a_1,a_2,G(x,y,a_1,a_2))\in \R^2\t\R; \frac{\partial G}{\partial x} = \frac{\partial G}{\partial y} = 0 \ at \ (x,y,a_1,a_2) \},$$
the discriminant of the family $G$ restrict to the singular curve
$\Sigma$ is given by
$$\mathcal{D}_2(F) = \{(a_1,a_2,H_1(t,a_1,a_2))\in \R^2\t\R; \frac{\partial H_1}{\partial t} = 0 \ at \ (t,a_1,a_2) \}$$
and the discriminant of the family $G$ restrict to the double point
curve $\Upsilon$ is the set
$$\mathcal{D}_3(F) = \{(a_1,a_2,H_2(t,a_1,a_2))\in \R^2\t\R; \frac{\partial H_2}{\partial t} = 0 \ at \ (t,a_1,a_2) \}.$$

If $F_1$ and $F_2$ are two $P$-$\Rl^+(X)$-equivalent deformations of
a germ $g$, then it is not difficult to show that the sets
$\mathcal{D}_i(F_1)$ and $\mathcal{D}_i(F_2)$ are diffeomorphics for
$i =1,2,3$. Therefore, it is enough to compute the sets
$\mathcal{D}_i(F)$ for the deformations given in Theorem
\ref{Teo:class}.\\

\np $\bullet$ The case $g(u,v,w) = \pm u$.

In this case, an $\Rl^+(X)$-versal deformation of $g$ is
$F(u,v,w,a_1,a_2) = \pm u$. Then the other families are
$$ G(x,y,a_1,a_2) = \pm x \ \ \ \ \ \ \ \ \ \ H_1(t,a_1,a_2) = \mp 6t^2 \ \ \ \ \ \ \ \ \ \ H_2(t,a_1,a_2) = \mp 2t^2.$$
Hence $\mathcal{D}_1(F)$ is the empty set and $\mathcal{D}_2(F) =
\mathcal{D}_3(F)$ is a plane.

Here, the fiber $g = 0$ is a plane transverse to both the
tangential line and the tangent cone of $X$.\\

\np $\bullet$ The case $g(u,v,w) = v + au^2, a\neq 0$.

In this case, an $\Rl^+(X)$-versal deformation is $F(u,v,w,a_1,a_2)
= v + au^2 + a_1u$. Then the other families are
\begin{itemize}
\item[] $G(x,y,a_1,a_2) = - 4y^3 - 2xy + ax^2 + a_1x,$

\item[] $H_1(t,a_1,a_2) = 8t^3 + 36at^4 - 6a_1t^2,$

\item[] $H_2(t,a_1,a_2) = 4at^4 - 2a_1t^2.$
\end{itemize}
Note that $H_2$ is a versal deformation of the boundary
$B_2$-singularity in the terminology of \cite{Arnold-Bk}. We have
\begin{itemize}
\item[] $\mathcal{D}_1(F) = \{(2y + 12ay^2,a_2,-4y^3 - 36ay^4)\}$,

\item[] $\mathcal{D}_2(F) = \{(a_1,a_2,0)\}\cup\{(2t + 12at^2,a_2,-4t^3 - 36at^4)\}$,

\item[] $\mathcal{D}_3(F) =
\{(a_1,a_2,0)\}\cup\{(4at^2,a_2,-4at^4)\}$.
\end{itemize}
These discriminants are illustrated in the Figure \ref{fig:disc02}.

\begin{figure}[h!]
\centering
\includegraphics[scale = 0.7]{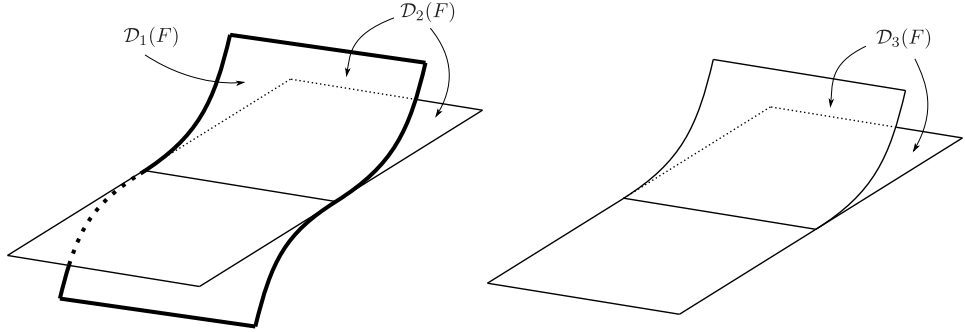}
\caption{\small{The discriminants $\mathcal{D}_2(F)$ and its subset
$\mathcal{D}_1(F)$ in bold (left) and the discriminant
$\mathcal{D}_3(F)$ (right) of $F = v + au^2 + a_1u$.}}
\label{fig:disc02}
\end{figure}

The tangent plane to the fiber $g = 0$ contains the tangential line
and is transverse to the tangent cone of $X$. The contact of the
tangential line with the fiber $g = 0$ is measured by the
singularities of $g(f(x,0)) = ax^2$ and is of type $A_1$.\\

\np $\bullet$ The case $g(u,v,w) = v + au^3, a\neq 0$.

In this case, an $\Rl^+(X)$-versal deformation is $F(u,v,w,a_1,a_2)
= v + au^3 + a_1u + a_2u^2$ and the other families are
\begin{itemize}
\item[] $G(x,y,a_1,a_2) = -4y^3 - 2xy + ax^3 + a_1x + a_2x^2,$

\item[] $H_1(t,a_1,a_2) = 8t^3 - 216at^6 - 6a_1t^2 + 36a_2t^4,$

\item[] $H_2(t,a_1,a_2) = - 8at^6 - 2a_1t^2 + 4a_2t^4.$
\end{itemize}
Note that $H_2$ is a versal deformation of the boundary
$B_3$-singularity in the terminology of \cite{Arnold-Bk}. Hence
\begin{itemize}
\item[] $\mathcal{D}_1(F) = \{(2y - 108ay^4 + 12a_2y^2,a_2,-4y^3 + 432ay^6 - 36a_2y^4)\}$,

\item[] $\mathcal{D}_2(F) = \{(a_1,a_2,0)\}\cup\{(2t - 108at^4 + 12a_2t^2,a_2,-4t^3 + 432at^6 - 36a_2t^4)\}$,

\item[] $\mathcal{D}_3(F) = \{(a_1,a_2,0)\}\cup\{(-12at^4 + 4a_2t^2,a_2,16at^6
-4a_2t^4)\}$.
\end{itemize}
See Figure \ref{fig:disc03}.

The second component of the discriminant $\mathcal{D}_3(F)$ is a
surface which is singular along the set
$\{(0,a_2,0)\}\cup\{(12at^4,6at^2,-8at^6)\}$. The singularity along
$(12at^4,6at^2,-8at^6)$ is a cuspidal edge when $t\neq 0$.

\begin{figure}[h!]
\centering
\includegraphics[scale = 0.7]{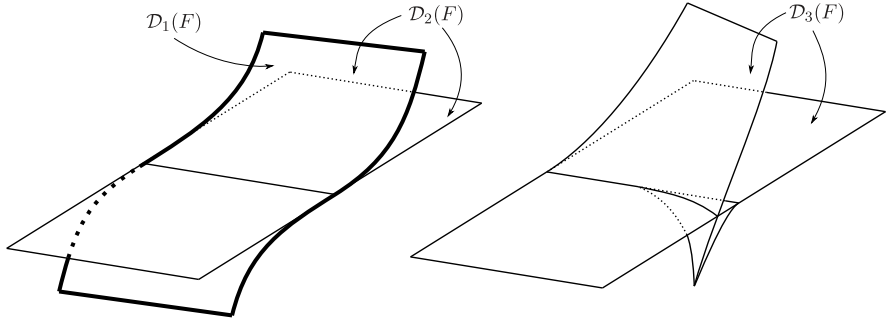}
\caption{\small{The discriminants $\mathcal{D}_2(F)$ and its subset
$\mathcal{D}_1(F)$ in bold (left) and the discriminant
$\mathcal{D}_3(F)$ (right) of $F = v + au^3 + a_1u + a_2u^2$. }}
\label{fig:disc03}
\end{figure}

Here, as in the previous case, the tangent plane to the fiber $g =
0$ contains the tangential line and is transverse to the tangent
cone of $X$. However the contact of the tangential line with the
fiber $g = 0$ is measured by the singularities of $g(f(x,0)) = ax^3$ and is of type $A_2$.\\

\np $\bullet$ The case $g(u,v,w) = \pm w + au^2 + bu^3, \ a\neq 0,
\pm \frac{1}{12}, \pm \frac{1}{4}, \ b\neq 0$.

In this case, an $\Rl^+(X)$-versal deformation is $F(u,v,w,a_1,a_2)
= \pm w + au^2 + bu^3 + a_1u + a_2v$, and the other families are
\begin{itemize}
\item[] $G(x,y,a_1,a_2) = \pm 3y^4 \pm xy^2 + ax^2 + bx^3 + a_1x - 4a_2y^3 - 2a_2xy,$

\item[] $H_1(t,a_1,a_2) = \mp 3t^4 + 36at^4 - 216bt^6 - 6a_1t^2 + 8a_2t^3,$

\item[] $H_2(t,a_1,a_2) = \pm t^4 + 4at^4 - 8bt^6-2a_1t^2.$
\end{itemize}

Therefore, the discriminant $\mathcal{D}_1(F)$ is the union of two
surfaces $S1, S2$, with $S1$ parametrised by
$$(x,y)\mapsto(\pm y^2 - 2ax -3bx^2, \pm y, \mp y^4 - ax^2 -2bx^3)$$
and $S2$ parametrised by
$$(a_2,t)\mapsto(\mp t^2 + 12at^2 - 108bt^4 + 2a_2t,a_2, \pm 3t^4 - 36at^4 + 432bt^6 - 4a_2t^3).$$
The first surface $S1$ is regular and its tangent plane at the
origin is $w=0$. The second surface $S2$ is singular along the curve
parametrised by $(\pm t^2-12at^2+324bt^4,\pm t-12at+216bt^3,\mp
t^4+12at^4-432bt^6)$. Using Corollary 1.5 in \cite{F-S-U-Y} we prove
that $S_2$ is a \textit{cuspidal cross cap} (that is, it is
$\mathcal{A}$-equivalent to the surface parametrised by
$(x,y^2,xy^3)$).

The intersection between these two components $S_1$ and $S_2$ is a
plane curve with a $Z_{17}$-singularity if $a = \mp\frac{1}{18}$
(that is, it is $\mathcal{R}$-equivalent to $x^3y+y^8+\lb xy^6$ for
some $\lb\in \R$) and a $Z_{13}$-singularity otherwise (that is, it
is $\mathcal{R}$-equivalent to $x^3y+y^6+\lb xy^5$ for some $\lb\in
\R$) . Therefore, this intersection is the image by the
parametrisation of the first component of two curves, which are, up
to diffeomorphisms, a line ($y=0$) and the zero-fiber of an $E_{12}$
singularity ($x^3 + y^7 + \delta xy^5 = 0$) if $a = \mp\frac{1}{18}$
and a line ($y=0$) and the zero-fiber of an $E_8$ singularity ($x^3
+ y^5 = 0$) otherwise.

The discriminant $\mathcal{D}_2(F)$ is the union of the plane
$\{(a_1,a_2,0)\}$ and the surface $S_2$ of $\mathcal{D}_1(F)$.

Finally, $\mathcal{D}_3(F) = \{(a_1,a_2,0)\}\cup\{(\pm t^2 + 4at^2 -
12bt^4, a_2, \mp t^4 - 4at^4 + 16bt^6)\}$.

The discriminants $\mathcal{D}_2(F)$ and $\mathcal{D}_3(F)$ are
illustrated in the Figure \ref{fig:disc04}.

\begin{figure}[h!]
\centering
\includegraphics[scale = 0.7]{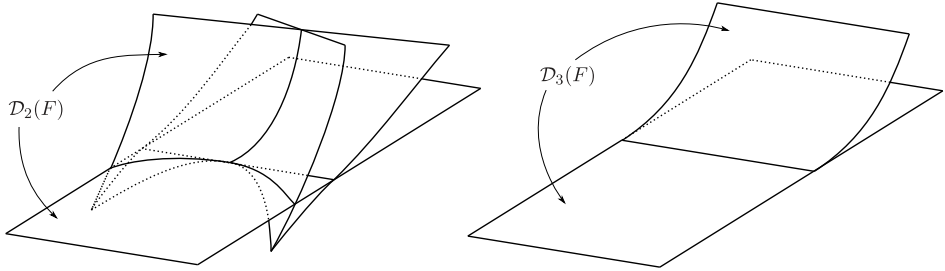}
\caption{\small{The discriminant $\mathcal{D}_2(F)$ (left) and the
discriminant $\mathcal{D}_3(F)$ (right) of $F =\pm w + au^2 + bu^3 +
a_1u + a_2v$. }} \label{fig:disc04}
\end{figure}

The tangent plane to the fiber $g = 0$ coincides with the tangent
cone of the swallowtail at the origin. The contact of the tangential
line with the fiber $g = 0$ is measured by the singularities of
$g(f(x,0)) = ax^2 + bx^3$ and is of type $A_1$.

\section{\bf The flat geometry of a swallowtail}

We use here the classification in \S 3 to study the flat geometry of
a geometric swallowtail $M$. The flat geometry is captured by the
contact of the geometric swallowtail $M$ with planes and is measured
by the singularities of the height function $h_\nu(p) = p\cdot \nu$,
with $\nu\in S^2$ orthogonal to the given plane. Varying $\nu$
locally in $S^2$ gives the family of height functions $H:M\t
S^2\rt\R$, given by $H(p,\nu) = h_\nu(p)$.

Let $g$ be a parametrisation of a geometric swallowtail. Then $g$ is
$\mathcal{A}$-equivalent to $f$ (the parametrisation of the standard
swallowtail). That is, there exist germs of diffeomorphisms
$\phi:(\R^2,0)\rt(\R^2,0)$ and $\psi:(\R^3,0)\rt(\R^3,0)$ such that
$g\circ \phi = \psi\circ f$.

We want to study the contact between the geometric swallowtail
$\psi(X)$ and the plane $h_\nu^{-1}(0)$ for some $\nu\in S^2$. This
contact is measured by the singularities of the function $h_\nu\circ
g:(\R^2,0)\rt (\R,0)$, but these singularities are the same as those
of the function $h_\nu\circ g\circ \phi = h_\nu\circ \psi\circ f$,
which in turn measure the contact between the standard swallowtail
$X = f(\R^2,0)$ and the surface $(h_\nu\circ\psi)^{-1}(0)$.

Note that if there exist another germs of diffeomorphisms $\phi_1$
and $\psi_1$ such that $g\circ \phi_1 = \psi_1\circ f$, then
$h_\nu\circ\psi_1 = h_\nu\circ \psi \circ(\psi^{-1}\circ \psi_1)$
and $\psi^{-1}\circ \psi_1(X) = \psi^{-1}\circ \psi_1\circ f(\R^2,0)
= \psi^{-1}\circ g\circ \phi_1(\R^2,0) =
f\circ\phi^{-1}\circ\phi_1(\R^2,0) = X$. The germ $\psi^{-1}\circ
\psi_1$ is a germ of diffeomorphism which preserves the standard
swallowtail $X$, that is, $\psi^{-1}\circ \psi_1\in \Rl(X)$.
Therefore, the function $h_\nu\circ \psi$ is well defined up to
elements in $\Rl(X)$ (see \cite{BW}).

Following the transversality theorem in the Appendix of \cite{BW},
for a generic swallowtail, the height functions $h_\nu$, for any
$\nu\in S^2$, can only have singularities of
$\Rl_e^+(X)$-codimension $\< 2$ at the origin. Furthermore, as the
height function $h_\nu:(\R^3,0)\rt (\R,0)$ is a submersion, the
function $h_v\circ\psi$ is also a submersion. Therefore
$h_\nu\circ\psi$ is $\Rl(X)$-equivalent to one of the normal forms
given in Theorem \ref{Teo:class}, that is, there exist a germ of
diffeomorphism $\varphi:(\R^3,0)\rt(\R^3,0)$ which preserves the
standard swallowtail $X$ such that $h_\nu\circ\psi =
\widehat{g}\circ\varphi$, where $\widehat{g}$ is one of the normal
forms given in Theorem \ref{Teo:class}. Hence the contact between a
geometric swallowtail $\psi(X)$ and the plane $h_\nu^{-1}(0)$
coincide with the contact of the standard swallowtail $X$ and the
fiber $\widehat{g}^{-1}(0)$ ( which is measured by the singularities
of the function $\widehat{g}\circ f$).

We have the following consequences about the flat geometry of a
generic swallowtail, where tangent/transverse to the swallowtail
(resp. singular curve and double point curve) means
tangent/transverse to its tangent cone (resp. the tangential line).

\begin{proposition}
The possible singularities of\, $\widehat{g}\circ f$ have the
following geometric interpretations:
\begin{itemize}
\item[(i)] $\pm u$ : the corresponding plane is transverse to both the swallowtail, its singular curve and its double point curve;

\item[(ii)] $v + au^2$ : the plane is transverse to the
swallowtail and is in the pencil of planes obtained as limiting
tangents to the double point curve (which coincide with that of the
singular curve);

\item[(iii)] $v + au^3$ : the plane is transverse to the
swallowtail and is in the pencil of planes obtained as limiting
tangents to the singular curve and is the limiting osculating plane
to the double point curve;

\item[(iv)] $\pm w + au^2 + bu^3$: the plane is the tangent cone of the swallowtail.
\end{itemize}
\end{proposition}
\begin{proof}
The proof follows form the analysis made in \S 4 for each case. \fim

\end{proof}

Consider a generic swallowtail with a parametrisation $g$ and let
$\lb$ and $\g$ be parametrisations of its singular curve and its
double point curve, respectively. For the family of height functions
$H$ we define
\begin{itemize}
\item[] $\mathcal{D}_1(H) = \{(\nu,h_\nu\circ g(x,y))\in S^2\t\R\, ; \displaystyle\frac{\partial h_\nu\circ g}{\partial x} = \frac{\partial h_\nu\circ g}{\partial y} = 0 \ {\rm at} \
(x,y,\nu)\}$;

\item[] $\mathcal{D}_2(H) = \{(\nu,h_\nu\circ \lb(t))\in S^2\t\R\, ; \displaystyle\frac{\partial h_\nu\circ \lb}{\partial t} = 0 \ {\rm at} \
(t,\nu)\}$;

\item[] $\mathcal{D}_3(H) = \{(\nu,h_\nu\circ \g(t))\in S^2\t\R\, ; \displaystyle\frac{\partial h_\nu\circ \g}{\partial t} =  0 \ {\rm at} \
(t,\nu)\}$.
\end{itemize}
The sets $\mathcal{D}_1(H), \mathcal{D}_2(H)$ and $\mathcal{D}_3(H)$
corresponds to the duals of the swallowtail, the singular curve and
the double point curve, respectively.

As discussed in the beginning of \S 5, the contact between a
swallowtail and a plane $h_\nu^{-1}(0)$ is described by that of the
fiber $\widehat{g}=0$ with the standard swallowtail, with
$\widehat{g}$ as in Theorem \ref{Teo:class}. Using this fact we can
show that $\mathcal{D}_i(H)$ is diffeomorphic to $\mathcal{D}_i(F)$,
for $i = 1,2,3$, where $F$ is an $\Rl^+(X)$-versal deformation of
$\widehat{g}$ with $2$-parameters. Therefore, the calculations and
figures in $\S 3.2$ give models, up to diffeomorphisms, of
$\mathcal{D}_i(H)$ for $i=1,2,3$.

\section*{Acknowledgements}

%

The author is supported by the FAPESP doctoral grant 2015/16177-2.

\let\thefootnote\relax\footnotetext{\\ Alex Paulo Francisco\\
Departamento de Matem\'{a}tica, ICMC-Universidade de S\~{a}o Paulo, 13560-970,\\
S\~{a}o Carlos, S\~{a}o Paulo, Brazil,\\
Email: alexpf@usp.br}

\end{document}